\documentclass[12pt,a4paper]{amsart}
\usepackage[unicode]{hyperref}
\usepackage{amsmath,amssymb,amsthm,braket,xfrac}
\usepackage[utf8]{inputenc}
\usepackage[french,english]{babel}
\usepackage{csquotes}
\usepackage[backend=biber,firstinits=true,sortcites=true,url=false,language=auto,autolang=other]{biblatex}
\usepackage[backgroundcolor=white,linecolor=black,textsize=scriptsize]{todonotes}
\newtheorem{theorem}{Theorem}
\newtheorem{conjecture}{Conjecture}
\newtheorem{lemma}{Lemma}
\newtheorem{corollary}{Corollary}
\theoremstyle{remark}
\newtheorem{remark}{Remark}

\theoremstyle{definition}

\DeclareMathOperator{\interior}{int}
\DeclareMathOperator{\Tr}{Tr}
\DeclareMathOperator{\var}{var}
\DeclareMathOperator{\id}{id}

\DeclareMathOperator{\dist}{dist}

\newcommand{\bbN}{\mathbb{N}}
\newcommand{\bbZ}{\mathbb{Z}}
\newcommand{\bbC}{\mathbb{C}}
\newcommand{\bbQ}{\mathbb{Q}}
\newcommand{\bbR}{\mathbb{R}}
\newcommand{\eps}{\varepsilon}

\author{Nataliya Goncharuk, Konstantin Khanin, and Yury Kudryashov}
\date{\today}
\title{Circle homeomorphisms with breaks with no \(C^{2-\nu}\) conjugacy}
\hypersetup{
  pdfauthor={Nataliya Goncharuk, Konstantin Khanin, and Yury Kudryashov},
  pdftitle={Circle homeomorphisms with breaks with no \(C^{2-\nu}\) conjugacy},
  pdfkeywords={},
  pdfsubject={},
  pdflang={English}}

\begin{document}

\maketitle

\begin{abstract}
  The rigidity theory for circle homeomophisms with breaks was studied intensively in the last 20 years.
  It was proved~\cite{Khanin_Khmelev_2003,Khanin_Teplinsky_2007,Khanin_Kocic_2017,Khanin_Kocic_Mazzeo_2017} that under mild conditions of the Diophantine type on the rotation number any two \(C^{2+\alpha}\) smooth circle homeomorphisms with a break point are \(C^1\) smoothly conjugate to each other, provided that they have the same rotation number and the same size of the break.
  In this paper we prove that the conjugacy may not be \(C^{2-\nu}\) even if the maps are analytic outside of the break points.
  This result shows that the rigidity theory for maps with singularities is very different from the linearizable case of circle diffeomorphisms where conjugacy is arbitrarily smooth, or even analytic, for sufficiently smooth diffeomorphisms.
\end{abstract}
\section{Introduction}
The rigidity theory for circle homeomorphisms with singularities and closely related renormalization theory for such maps was a subject of intensive studies in the last 20 years~\cite{Avila_2013,Guarino_Martens_Welington_2018,Estevez_Faria_Guarino_2018,Gorbovickis_Yampolsky_2020,Yampolsky_2002,Yampolsky_2003,deFaria_deMelo_1999,deFaria_deMelo_2000,Cunha_Smania_2013,Cunha_Smania_2014,Khanin_Kocic_2018}.
In majority of the papers the following two classes of singularities were studied:
critical circle maps, where a map \(f\) is smooth but the derivative vanishes at one point \(f'(x_{cr})=0\), and circle maps with breaks, where a map \(f\) is smooth outside of one point \(x_{br}\) and it has a jump discontinuity of the first derivative at \(x_{br}\).
The type of singularity can be characterized by the order \(\gamma>1\) of the critical point, or by the size of the break \(c=f'(x_{br}-0)/f'(x_{br}+0)\).
The main result of the rigidity theory can be formulated in the following way:
if two maps \(f,g\) are topologically conjugate to each other, that is they have the same irrational rotation number \(\rho\), and they also have the same type of singularity, namely the same \(\gamma\), or the same \(c\), then the conjugacy which maps the singular point into the singular point is \(C^1\) smooth, provided the maps \(f,g\) are sufficiently smooth outside of their singularities.
For circle maps with a break point, one also has to impose some mild Diophantine conditions on the rotation number \(\rho\).
For critical maps, it is conjectured, and in many cases proved, that Diophantine conditions for \(C^1\) rigidity are not needed.
For a more restrictive class of rotation numbers, which includes irrational numbers of bounded type, one can prove that the conjugacy is \(C^{1+\epsilon}\) smooth for some \(\epsilon>0\), see~\cite{deFaria_deMelo_1999,deFaria_deMelo_2000}.
We should add that all the rigidity results in the critical case are proved only for \(\gamma\) being an odd integer number greater than 1, although it is generally believed that the rigidity statements remain valid for all \(\gamma>1\).

The rigidity theory is closely related to renormalization theory.
One first has to prove the exponential convergence of renormalization which is the crucial step in establishing rigidity.
Notice that convergence of renormalization normally holds for all irrational rotation numbers, and Diophantine conditions are only used to prove rigidity results.

The rigidity theory for maps with singularities is also related to the linearization theory (Herman theory) for smooth circle diffeomophisms~\cite{He79,Yoccoz_1984}.
The linearization theory says that any smooth enough circle diffeomorphism with a typical irrational rotation number~\(\rho\) can be smoothly linearized, that is smoothly conjugate to the linear rotation \(f_\rho\colon x\mapsto x+\rho\) by the angle \(\rho\).
Of course, the linear rotation here is just the simplest representative of the class of smooth diffeomophisms with rotation number~\(\rho\).
As above, typical rotation number means that certain Diophantine conditions on \(\rho\) are satisfied.
The important feature of the linearization theory is the fact that a conjugacy can be shown to be arbitrarily smooth if the map \(f\) is smooth enough.
Usually one must first prove the \(C^1\) smoothness of the conjugacy, and then upgrade it to higher smoothness using the smoothness of the map.
Such an upgrade is well understood, and is based on some kind of bootstrap technique.
To be more precise let us formulate the following theorem due to Katznelson and Ornstein~\cite{Katznelson_Ornstein_1989}.
We say that the rotation number \(\rho\) belongs to the Diophantine class \(D_\delta\) if there exists a constant \(C>0\) such that for all nonzero \(q\in \bbZ\) and \(p\in \bbZ\)
\[
  |q\rho-p|\geq \frac{C}{q^{1+\delta}}.
\]
\begin{theorem}
  [Katznelson, Ornstein, 1989]
  Let \(f\in C^k, k>2, k\in\bbR\) and \(\rho=\rho(f) \in D_\delta\).
  Assume that \(k-1-\delta>1\).
  Then \(f\) is linearizable,
  \[
    h\circ f\circ h^{-1}= f_\rho,
  \]
  and the conjugacy \(h\) is \(C^{k-1-\delta -\epsilon}\)-smooth for arbitrarily small \(\epsilon>0\).
\end{theorem}
It was proved in several cases~\cite{Khanin_Teplinsky_2009,Sinai_Khanin_1989}, and is likely to be true in general (D. Ornstein, private communication) that arbitrarily small \(\epsilon\) can be removed from the theorem.

In view of the above behaviour in the diffeomorphism case it is natural to ask whether higher smoothness of the conjugacy holds in the case of circle maps with singularities.
In this paper we prove that it is not the case for maps with breaks.
Namely, we prove that for any break size \(c\neq 1\) there exists a positive constant \(\nu(c)<1\) such that for an arbitrary irrational rotation number \(\rho\) there exist analytic maps \(f\), \(g\) with the break size \(c\) and rotation number \(\rho\) which are not \(C^{2-\nu(c)}\) conjugate to each other.
We also conjecture that similar statement holds for critical circle maps.

Related questions were studied in~\cite{Khanin_Kocic_2013} and~\cite{Kocic_2016}.
In the first paper it was shown that robust rigidity, which was established in the case of critical circle maps~\cite{Khanin_Teplinsky_2007}, does not hold for maps with a break point.
In other words, \(C^{1}\) rigidity requires certain conditions of the Diophantine type. In~\cite{Kocic_2016} it was proved that for Lebesgue generic rotation numbers \(C^{1+\alpha}\) rigidity does not hold.
The proof uses the growth rate of partial quotients $a_n$ in the continued fraction representation for $\rho=[a_1,a_2, \dots, a_n, \dots]$.
Our approach is completely different.
In particular, the main result holds for rotation numbers of bounded type for which the conjugacy is $C^{1+\alpha}$ smooth for some $\alpha>0$, but it does not belong to the class $C^{2-\nu}$, as we show in this paper.
Below we present a new general obstruction mechanism for higher smoothness rigidity.
This new mechanism allows us to prove the main result for all irrational rotation numbers.

In the next five sections we formulate and prove the main theorem.
We finish with concluding remarks and open problems in \autoref{sec:concluding}.

\section{Statement of the main theorem}
We will say that a circle homeomorphism is of class \(B^r(p, c)\), \(r=1,\dotsc,\infty,\omega\), if it is \(C^r\) smooth everywhere except at the unique break point \(x_{br}=p\), and the break size \(f'(p-0)/f'(p+0)\) equals \(c\).
We shall always assume that \(c\neq 1\).
Let \(B^{r}_{irr}(p, c)\) be the set of circle homeomorphisms \(f\in B^{r}(p, c)\) that have irrational rotation numbers.

Denote by \(S(f)\) the Schwarzian derivative of~\(f\):
\[
  S(f)(x)=\frac{f'''(x)}{f'(x)}-\frac{3}{2}\left(\frac{f''(x)}{f'(x)}\right).
\]

The main result of this paper is the following theorem and its corollary.
\begin{theorem}%
  \label{thm:main}
  For every \(c\ne 1\) there exists a number \(\nu(c)\in(0, 1)\) such that the following property holds.
  Let \(f \in  B^3(p, c)\) be a circle homeomorphism  with an irrational rotation number \(\rho\) and negative Schwarzian derivative \(S(f)\) at all points of smoothness.
  Let \(g\) be the unique linear fractional circle map of the class \(B^{3}(p, c)\) with the rotation number \(\rho\).
  Then \(f\) is not \(C^{1+\alpha}\) conjugate to \(g\) for any \(\alpha>1-\nu(c)\).
\end{theorem}

\begin{remark}
  Here and below a map~\(h\) is said to be \(C^{1+\alpha}\), \(0<\alpha\le 1\), if it is continuously differentiable and its derivative is Hölder continuous with exponent~\(\alpha\).
\end{remark}

\begin{remark}
  Note that the constant \(\nu(c)\) depends only on \(c\neq 1\), and does not depend on \(f,g\) or \(\rho\).
\end{remark}
\begin{corollary}%
  \label{cor:no-conj-lin-frac}
  Consider a linear fractional circle homeomorphism \(g\in B^{\omega}_{irr}(p, c)\).
  Then there exists a family of circle homeomorphisms \(f_\eps \in  B^\omega(p, c)\), \(\eps \in \bbR\), such that \(f_{0}=g\), \(f_\eps\) is continuous in \(\eps\) in \(C^\infty([p, p+1])\) topology, \(\rho(f_\eps)=\rho\) for all  \(\eps \in \bbR\), but none of the maps \(f_\eps\), \(\eps\neq0\), are \(C^{1+\alpha}\) conjugate to \(g\) with \(\alpha>1-\nu(c)\).

  Moreover, \(f_\eps\) continuously depends on \(\eps\) as an analytic map of some neighborhood of \([p, p+1]\) in \(\bbC\).
\end{corollary}

Let us deduce \autoref{cor:no-conj-lin-frac} from \autoref{thm:main}.
\begin{proof}
  [Proof of \autoref{cor:no-conj-lin-frac}.]
  Without loss of generality we may and will assume that \(p=0\).

  Let \(\bar f_\eps\in B^{\omega}(0, e^{\eps})\) be given by \(\bar f_\eps(x)=\frac{e^{\eps x}-1}{e^\eps-1}\) for \(\eps\neq0\) and \(\bar f_0(x)=x\).
  Note that \(\bar f_\eps\) maps \(0\) to \(0\), \(1\) to \(1\), continuously depends on \(\eps\), and has Schwarzian derivative \(S(\bar f_\eps)=S(e^{\eps x})=-\frac{\eps^{2}}{2}<0\).

  Let \(f_\eps\) be given by \(f_\eps=R_\eps \circ \bar f_\eps\), where \(R_\eps\) is a fractional linear map that we shall choose later.
  Note that, since \(S(R_\eps)=0\), we have \(S(f_\eps)=S(\bar f_\eps)=-\frac{\eps^{2}}{2}<0\) for all \(\eps\neq0\).
  Hence, by \autoref{thm:main} it suffices to find \(R_\eps\) such that that \(f_\eps\in B^\omega(0, c)\), \(\rho (f_\eps)=\rho (g)\), and \(f_\eps\) is continuous in \(\eps\).
  Equivalently, we need a continuous family of fractional linear maps \(R_\eps\), \(R_0=\id\), such that \(R_\eps(1)=R_\eps(0)+1\), \(\frac{R_\eps'(1)}{R_\eps'(0)}=ce^{-\eps}\), and \(\rho(R_\eps\circ \bar f_\eps)=\rho (g)\).
  The first two conditions define \(R_\eps\) up to an additive constant, and there is a unique choice of such a constant that leads to the prescribed rotation number \(\rho (g)\).
  Continuity of \(R_\eps\) in \(\eps\) easily follows from the continuity of \(\bar f_\eps\) and standard properties of the rotation number.
\end{proof}

\section{Schwarzian derivative as a mixed partial derivative}\label{sec:notation}

To avoid considering cases or adding unnecessary inequalities as assumptions, denote by \([x..y]\), \(x, y \in \bbR\), the interval \([\min(x, y), \max(x, y)]\) joining \(x\) to \(y\).

Given a smooth map \(f\colon I\to\bbR\), put
\begin{equation}
  \label{eq:tilde-f}
  \tilde f(x, y)=
  \begin{cases}
    \ln\frac{f(y)-f(x)}{y-x}, & \text{if \(x \ne y\);}\\
    \ln f'(x),                & \text{if \(x=y\).}
  \end{cases}
\end{equation}
For \(f\in B^{r}(p,c)\), the function \(\tilde f\) is defined on \(I\times I\), \(I=[p, p+1]\).
For \(x=y=p\) and \(x=y=p+1\) we take the right and the left derivative of \(f\), respectively.

Function~\(\tilde f\) describes how does \(f\) change the length of the interval \([x..y]\) in the logarithmic scale.
Given four points \(x_{1}, x_{2}, x_{3}, x_{4}\), a map \(f\) changes the logarithm of their cross-ratio by
\begin{multline}
  \ln\left(\frac{f(x_{4})-f(x_{3})}{f(x_{3})-f(x_{1})}:\frac{f(x_{4})-f(x_{2})}{f(x_{2})-f(x_{1})}\right)-
  \ln\left(\frac{x_{4}-x_{3}}{x_{3}-x_{1}}:\frac{x_{4}-x_{2}}{x_{2}-x_{1}}\right)=\\
  \tilde f(x_{3}, x_{4})-\tilde f(x_{3}, x_{1})+\tilde f(x_{2}, x_{1})-\tilde f(x_{2}, x_{4}).
\end{multline}

In particular, the limit of the left hand side as \(x_{3}\to x_{4}\) and \(x_{2}\to x_{1}\) exists and is equal to
\begin{multline}
  \lim_{\substack{x_{3}\to x_{4}\\x_{2}\to x_{1}}}\ln\left(\frac{f(x_{4})-f(x_{3})}{f(x_{3})-f(x_{1})}:\frac{f(x_{4})-f(x_{2})}{f(x_{2})-f(x_{1})}\right)-
  \ln\left(\frac{x_{4}-x_{3}}{x_{3}-x_{1}}:\frac{x_{4}-x_{2}}{x_{2}-x_{1}}\right)=\\
  \tilde f(x_{1}, x_{1})+\tilde f(x_{4}, x_{4})-2\tilde f(x_{1}, x_{4}) =: \xi_{f}(x_{1}, x_{4}).
\end{multline}

Since \(\xi_{f}\) is a symmetric function, we will use notation \(\xi_{f}(J)\) for \(\xi_{f}(a, b)\), \(J=[a, b]\).
Our proof of \autoref{thm:main} will be based on some estimates on \(\xi_{f}\), and we will use the following simple formulas for \(\widetilde{f_{1} \circ f_{2}}\) and \(\xi_{f_{1}\circ f_{2}}\).
\begin{align}
  \label{eq:F-comp}
  \widetilde{f_{1} \circ f_{2}}(x, y)&=\tilde f_{1}(f_{2}(x), f_{2}(y))+\tilde f_{2}(x, y),\\
  \label{eq:xi-comp}
  \xi_{f_{1}\circ f_{2}}(J) &=\xi_{f_{1}}(f_{2}(J))+\xi_{f_{2}}(J).
\end{align}

It is easy to see that linear fractional maps preserve cross-ratio.
Hence, \(\xi_f(J)=0\) if \(f\) is a linear fractional map.
If \(f\) is \(C^{2}\) smooth on~\(I\), then \(\tilde f\) is \(C^{1}\) smooth on \(I\times I\).
If \(x\ne y\), then the mixed partial derivative of \(\tilde f\) is given by
\[
  \frac{\partial^{2}\tilde f}{\partial x\partial y}=\frac{f'(x)f'(y)}{{(f(y)-f(x))}^{2}}-\frac{1}{{(y-x)}^{2}}.
\]
For a \(C^{3}\) smooth function \(f\) this function extends to a continuous function on \(I\times I\), and on the diagonal we have
\begin{equation}
  \label{eq:Sf-mixed-partial}
  \left.\frac{\partial^{2}\tilde f}{\partial x\partial y}\right|_{y=x}=\frac{1}{6}S(f)(x),
\end{equation}
where
\[
  S(f)(x)=\frac{f'''(x)}{f'(x)}-\frac{3}{2}{\left(\frac{f''(x)}{f'(x)}\right)}^{2}
\]
is the Schwarzian derivative of~\(f\).

For analytic maps~\eqref{eq:Sf-mixed-partial} is stated, e.g., in~\cite{Hawley_Schiffer_1966} but it is easy to verify~\eqref{eq:Sf-mixed-partial} for any \(C^{3}\) smooth function \(f\) by substituting Taylor series expansions for \(f(y)\) and \(f'(y)\).

In particular, for \(f\) from \autoref{thm:main}, for some \(s>0\) we have
\begin{equation}
  \label{eq:dF-lt}
  \frac{\partial^{2}\tilde f}{\partial x\partial y}<-s
\end{equation}
in some neighborhood of the diagonal \(x=y\).
Indeed, take \(s\) such that \(S(f)<-6s\) everywhere, then~\eqref{eq:Sf-mixed-partial} implies~\eqref{eq:dF-lt}.
\section{Preliminary estimates}\label{sec:estimates}
In this section we remind some well-known estimates on the behaviour of the cross-ratio under smooth conjugacy.
For completeness we provide proofs of these estimates.
We also restate them in terms of \(\tilde f\) and \(\xi_{f}\) introduced above.

We start with the following observation.
\begin{lemma}%
  \label{lem:C2-hom}
  Let \(h\colon I\to\bbR\) be a \(C^{1+\alpha}\) smooth map with positive derivative, \(0<\alpha\le 1\).
  Then for any \(J \subset I\) we have \(\xi_{h}(J)=O\left({|J|}^{\alpha}\right)\) as \(|J|\to 0\).
  Moreover, if \(h\) is a \(C^{2}\) smooth function, then \(|\xi_{h}(J)|=o(|J|)\) as \(|J|\to0\).
\end{lemma}
\begin{proof}
  Recall that for \(J=[x,y]\subset I\) we have
  \[
    \xi_{h}(J)=\tilde h(x, x)+\tilde h(y, y)-2\tilde h(x, y)=\ln h'(x)+\ln h'(y)-2\ln\frac{h(y)-h(x)}{y-x}.
  \]

  The first inequality immediately follows from the Mean Value Theorem.
  Take \(a\in[x,y]\) such that \(h'(a)=\frac{h(y)-h(x)}{y-x}\).
  Since \(h'>0\) on \(I\), the function \(\ln h'\) is Hölder continuous as well, thus we have
  \begin{align*}
    |\xi_{h}(J)|&=|\ln h'(x)+\ln h'(y)-2\ln h'(a)|\\
                &\le|\ln h'(x)-\ln h'(a)|+|\ln h'(y)-\ln h'(a)|\\
                &=O\left({|x-a|}^{\alpha}\right)+O\left({|y-a|}^{\alpha}\right)\\
                &=O\left({|x-y|}^{\alpha}\right).
  \end{align*}

  Now let us prove that for a \(C^{2}\) smooth map \(h\), \(h'>0\), we have \(\xi_{h}(J)=o(|J|)\).
  Since \(\ln h'\) is a \(C^{1}\) smooth function, we have
  \[
    \ln h'(x)+\ln h'(y)=2\ln h'\left(\frac{x+y}{2}\right)+o(|J|),
  \]
  thus it suffices to show that
  \[
    \ln h'\left(\frac{x+y}{2}\right)-\ln\frac{h(y)-h(x)}{y-x}=o(|J|).
  \]
  Since \(\ln\) is an analytic function and both arguments are bounded away from \(0\) and \(\infty\), it suffices to show that
  \[
    h'\left(\frac{x+y}{2}\right)-\frac{h(y)-h(x)}{y-x}=o(|J|).
  \]
  This estimate holds for any \(C^{2}\) smooth function \(h\), and it can be easily proved by rewriting \(h(x)\) and \(h(y)\) using Taylor's theorem with center \(\frac{x+y}{2}\) and Lagrange form of the remainder.
\end{proof}
\begin{corollary}%
  \label{cor:C2-conj}
  Let \(f\) and \(g\) be two piecewise differentiable circle homeomorphisms conjugate by a diffeomorphism \(h\), \(h \circ f = g \circ h\).
  Then for  \(n \in \bbN\) and a closed interval \(J\) we have
  \begin{align*}
    \left|\xi_{f^{n}}(J)-\xi_{g^{n}}(h(J))\right|&=O\left({|J|}^{\alpha}+{|f^{n}(J)|}^{\alpha}\right),&\text{if }h&\in C^{1+\alpha}(S^{1});\\
    \left|\xi_{f^{n}}(J)-\xi_{g^{n}}(h(J))\right|&=o\left(|J|+|f^{n}(J)|\right),&\text{if }h&\in C^{2}(S^{1}).
  \end{align*}
  In particular, if \(f\) and \(g\) are as in \autoref{thm:main} and \(p\notin \bigcup_{k=0}^{n-1}f^{k}(J)\), then \(\xi_{g^{n}}(h(J))=0\), and we have
  \begin{align*}
    \left|\xi_{f^{n}}(J)\right|&=O\left({|J|}^{\alpha}+{|f^{n}(J)|}^{\alpha}\right),&\text{if }h&\in C^{1+\alpha}(S^{1});\\
    \left|\xi_{f^{n}}(J)\right|&=o\left(|J|+|f^{n}(J)|\right),&\text{if }h&\in C^{2}(S^{1}).
  \end{align*}
  Here all estimates hold as \(|J|\to 0\) uniformly in \(n\).
\end{corollary}
\begin{proof}
  The first part of the corollary follows from~\autoref{lem:C2-hom}, \(h \circ f^{n}=g^{n} \circ h\), and~\eqref{eq:xi-comp}.

  In order to prove the second part, note that the conjugacy \(h\) maps the break point \(p\) to itself, hence we have \(p\notin \bigcup_{k=0}^{n-1}g^{k}(h(J))\).
  Therefore, \(g^{n}\) is a linear fractional map on \(h(J)\), thus \(\xi_{g^{n}}(h(J))=0\), and the estimates follow from the first part.
\end{proof}

Due~to~\eqref{eq:xi-comp}, we can rewrite \(\left|\xi_{f^{n}}(J)\right|\) as a sum,
\begin{equation}
  \label{eq:xi-pow}
  \left|\xi_{f^{n}}(J)\right|=\left|\sum_{k=0}^{n-1}\xi_{f}(f^{k}(J))\right|.
\end{equation}

The following lemma provides a negative upper bound for each summand \(\xi_{f}(f^{k}(J))\).
\begin{lemma}%
  \label{lem:DLCR-S}
  Consider a circle homeomorphism \(f \in B^3(p, c) \) such that \(S(f)<0\) on \([p,p+1]\).
  Then there exist positive \(\eps\) and \(s\) such that for any interval \(J\subset[p, p+1]\), \(|J|<\eps\) we have
  \[
    \xi_{f}(J)\le -s|J|^{2}.
  \]
\end{lemma}

\begin{proof}
  Take \(s>0\) and \(\eps>0\) such that \(\frac{\partial^{2}\tilde f}{\partial x\partial y}(x, y)\le -s\) for any \(x, y\in[p, p+1]\), \(|y-x|<\eps\), see~\eqref{eq:dF-lt} for details.
  Now consider an interval \(J=[x,y]\subset[p, p+1]\), \(|J|<\eps\), and apply the Fundamental Theorem of Calculus twice,
  \begin{align*}
    \xi_{f}(J)&=\tilde f(x, x)+\tilde f(y, y)-\tilde f(x, y)-\tilde f(y, x)\\
              &=\int_{x}^{y}\int_{x}^{y}\frac{\partial^{2}\tilde f}{\partial x\partial y}(a, b)\;da\;db\\
              &\le -s{|J|}^{2}.
  \end{align*}
\end{proof}

Summarizing \autoref{cor:C2-conj},~\eqref{eq:xi-pow}, and \autoref{lem:DLCR-S}, we get the following statement.
\begin{lemma}%
  \label{lem:sum-sqr-le}
  In the settings of \autoref{thm:main}, there exists \(\eps>0\) such that the following holds.
  Consider a number \(n \in \bbN\) and a closed interval \(J\) such that \(p\notin\bigcup_{k=0}^{n-1}\interior\left(f^{k}(J)\right)\) and \(|f^{k}(J)|\le \eps\), \(0\le k<n\).
  Then
  \begin{align*}
    \sum_{k=0}^{n-1}{\left|f^{k}(J)\right|}^{2}&=O\left({|J|}^{\alpha}+{|f^{n}(J)|}^{\alpha}\right), &\text{if }h&\in C^{1+\alpha}(S^{1});\\
    \sum_{k=0}^{n-1}{\left|f^{k}(J)\right|}^{2}&=o\left(|J|+|f^{n}(J)|\right), &\text{if }h&\in C^{2}(S^{1})
  \end{align*}
  as \(\max(|J|, |f^{n}(J)|)\to0\).
  Here both estimates are uniform in \(n\).
\end{lemma}

\section{Decay of the intervals of dynamical partitions}%
\label{sec:decay-interv-part}

In this section we prove some estimates on the lengths of the intervals of dynamical partitions for a circle homeomorphism with a break.
We use convergence of renormalizations to strengthen some well-known estimates on these lengths: we show that bounds depend only on \(c\), not on \(f\) or \(\rho(f)\).
While the proof of \autoref{thm:main} only uses these estimates for \(f\in B^3_{irr}(p, c)\), we state and prove them for a larger class of homeomorphisms.
Namely, put
\[
  B^{2+}_{irr}(p, c)=\bigcup_{\eps>0}B^{2+\eps}_{irr}(p, c).
\]

Let \(f\) be a circle homeomorphism with an irrational rotation number \(\rho = [a_1,a_2,\dots,a_n, \dots ]\).
Recall that the sequence of dynamical partitions \(\zeta_n\) for~\(f\) with a base point \(p\) is defined in the following way.
Put \(\Delta_n^0=[p..f^{q_n}(p)]\), \(\Delta_{n}^{k}=f^{k}\Delta_{n}^{0}\), \(n, k \in \bbN\), where \(q_n\) are the denominators of the continued fraction convergents \(\frac{p_n}{q_n}=[a_1,a_2,\dots,a_n]\).
Fix \(n\) and consider two sequences of intervals \(\set{\Delta_{n-1}^k | 0\leq k<q_n}\) and \(\set{\Delta_{n}^m | 0\leq m<q_{n-1}}\).
It is a simple combinatorial fact~\cite{Sinai_Khanin_1989} that together these two sequences form a partition of the unit circle, which we denote by \(\zeta_n\).
In particular, the interiors of the elements of \(\zeta_{n}\) are pairwise disjoint.

The main result of this section is the following theorem.
It says that all intervals of \(\zeta_{n}\) are exponentially small, and the shortest of the intervals \(\Delta_{n-1}^{k}\in\zeta_{n}\), \(0\le k<q_{n}\), is exponentially smaller than its invariant measure.
\begin{theorem}%
  \label{thm:zeta-decay}
  There exist universal constants \(\gamma_{1}(c), \gamma_{2}(c)\in (0, 1)\) such that for any circle homeomorphism \(f\in B^{2+}_{irr}(p, c)\), the following asymptotic estimates hold:
  \begin{align}
    \label{eq:zeta-max-delay}
    \max_{\Delta\in\zeta_{n}}|\Delta|&=O({\gamma_{1}(c)}^{n});\\
    \label{eq:zeta-min-delay}
    \min_{0\le k<q_{n}}|\Delta_{n-1}^{k}|&=O({\gamma_{2}(c)}^{n}\mu_{n-1}),
  \end{align}
  where \(\mu_{n-1}=|q_{n-1}\rho-p_{n-1}|\) is the invariant measure of each \(\Delta_{n-1}^{k}\).
\end{theorem}
Here the constants behind the \(O(\cdot)\) notation may depend on \(f\) but \(\gamma_{1}\) and \(\gamma_{2}\) depend only on \(c\), not on \(f\) or its rotation number.
We prove the first estimate in \autoref{cor:exp-decay-max} and the second estimate in \autoref{small_intervals}.
Then we deduce \autoref{thm:main} from \autoref{thm:zeta-decay} in \autoref{sec:proof-main}.

\subsection{Renormalizations and derivatives of \(f^{k}\)}
Given a circle homeomorphism with an irrational rotation number, we define renormalized maps \(f_n\), \(g_n\) which correspond to two branches of the first return map onto the fundamental interval \(I_n=\Delta^0_{n-1}\cup \Delta^0_{n}\) expressed in the renormalized coordinate \(z\).
Coordinate \(z\) corresponds to the affine change of variables such that \(z(p)=0\), \(z(f^{q_{n-1}}p)=-1\).
It is easy to see that
\begin{align*}
  f_n(z)&= A_n^{-1}\circ f^{q_n}\circ A_n(z), &z&\in [-1,0],\\
  g_n(z)&= A_n^{-1}\circ f^{q_{n-1}}\circ A_n(z),& z&\in [0,\alpha_n],
\end{align*}
where \(A_n\) is the change of coordinates from \(z\) to \(x\) and \(\alpha_n=|\Delta_n^0|/|\Delta^0_{n-1}|\).
Note that \(g_{n}\) differs from \(f_{n-1}|_{[-\alpha_{n-1}\alpha_{n}, 0]}\) by a linear change of coordinate,
\begin{equation}
  \label{eq:gn-fn-1}
  g_{n}(z)=-\frac{1}{\alpha_{n-1}}f_{n-1}(-\alpha_{n-1}z)
\end{equation}

The pair~\(R^nf=(f_n(z),g_n(z))\) is called \emph{the \(n\)-th step renormalization} of~\(f\).
On each step the renormalization transformation \(R\) transforms \(R^nf\) into \(R^{n+1}f\) corresponding to the first return map to the fundamental interval \(I_{n+1}\).
Notice that~\(R^{n}f\) can be viewed as a circle homeomorphism \(Rf_{n}\colon \sfrac{[-1, \alpha_{n}]}{(-1\sim \alpha_{n})}\to \sfrac{[-1, \alpha_{n}]}{(-1\sim \alpha_{n})}\).
This circle homeomorphism has two breaks, at \(0\) and at \(-1\sim\alpha_{n}\).
It is easy to see that the product of the sizes of these breaks is equal to \(c\),
\begin{equation}
  \label{eq:mul-break}
  \frac{g_{n}'(0)}{f_{n}'(0)}\cdot\frac{f_{n}'(-1)}{g_{n}'(\alpha_{n})}=c.
\end{equation}
Indeed, \(f_{n}\circ g_{n}=g_{n}\circ f_{n}=A_{n}^{-1}\circ f^{q_{n}+q_{n-1}}\circ A_{n}\), and the product of break sizes is equal to the break size of \(f^{q_{n}+q_{n-1}}\) at zero.

The following result was proved in~\cite{Khanin_Vul_1991} and~\cite{Teplinskii_Khanin_2004}; see also~\cite{Ghazouani_Khanin_2021} for a simpler proof of the estimate in \(C^{1}\) metric.
\begin{theorem}
  [{\cites[Proposition 4.6 and Theorem 2]{Khanin_Vul_1991}[Assertions 1, 2]{Teplinskii_Khanin_2004}[Theorem 4]{Ghazouani_Khanin_2021}}]\label{thm:dist-fnFn}
  For any \(f\in B^{2+}_{irr}(p, c)\), \(c\ne 1\), there exist real numbers \(C>0\), \(0<\lambda<1\) and a sequence \(v_{n}\) such that
  \begin{equation}
    \dist_{C^2[-1,0]}(f_n, F_{\alpha_n,v_n,c_n}(z)) \leq C\lambda^n;\label{eq:fn-Fn}
  \end{equation}
  Here and below
  \begin{align}
    c_{n}&=c^{{(-1)}^{n}}, & F_{\alpha,v,c}(z)&=\frac{\alpha+\sqrt{c}z}{1-vz}.
  \end{align}
  Moreover, if \(n\) is large enough, then \((\alpha_n,v_n) \in U_{c_{n}}\), where
  \begin{align*}
    U_c&=[0, \sqrt{c}]\times \left[\frac{\sqrt{c}-1}2, \sqrt{c} -1\right]&\text{if }c&>1\\
    U_c&=[0, \sqrt{c}]\times \left[\sqrt{c} -1, \frac{\sqrt{c}-1}2 \right]&\text{if }c&<1.
  \end{align*}
\end{theorem}

Put \(\hat{c}=\max{(c,c^{-1})}\).
Next lemma provides a uniform convexity bound on \(f_{m}\) and \(g_{m}\) for large \(m\).
\begin{lemma}%
  \label{lem:universal_convexity}
  For any \(c\ne 1\), \(Q_2>2(\hat{c}^{3/2}-\hat{c})\), \(Q_1\in\left(0, \hat{c}^{-1}-\hat{c}^{-3/2}\right)\), and \(f\in B^{2+}_{irr}(p, c)\) for all \(m\) large enough we have
  \begin{subequations}%
    \label{eq:uniform-convexity}
    \begin{align}
      \label{eq:convexity-fm}
      Q_1&< |f_m''(z)|< Q_2, &z&\in [-1,0],\\
      \label{eq:convexity-gm}
      \alpha_{m-1}Q_1&< |g_m''(z)|< \alpha_{m-1}Q_2, &z&\in [0,\alpha_m].
    \end{align}
  \end{subequations}
  The second derivatives \(f_m''(z)\), \(g_m''(z)\) are positive if \(c_{m}>1\), and negative if \(c_{m}<1\).
\end{lemma}
\begin{proof}
  It suffices to prove~\eqref{eq:convexity-fm}, then~\eqref{eq:convexity-gm} will follow from~\eqref{eq:convexity-fm} and~\eqref{eq:gn-fn-1}.
  Since \(\dist_{C^{2}[-1, 0]}(f_{m}, F_{\alpha_{m}, v_{m}, c_{m}})\to 0\) as \(m\to\infty\), it suffices to prove that for any \((\alpha, v)\in U_{c}\), \(\hat c=\max(c, 1/c)\), the second derivative \(F_{\alpha, v, c}''=\frac{2v(\alpha v+\sqrt{c})}{{(1-vz)}^{3}}\) has the same sign as \((c-1)\) and its absolute value belongs to the interval \(\left[\hat{c}^{-1}-\hat{c}^{-3/2}, 2(\hat{c}^{3/2}-\hat{c})\right]\).
  These estimates follow immediately from the definition of \(U_{c}\).
\end{proof}

From now on, we fix some \(Q_{1}(c)\) and \(Q_{2}(c)\) that satisfy assumptions of \autoref{lem:universal_convexity}.
Note that for sufficiently large \(m\) we have \(\alpha_{m-1}\le \sqrt{c_{m-1}}\), thus~\eqref{eq:convexity-gm} implies a uniform upper bound on the second derivative of \(g_{m}\),
\[
  |g_{m}''(z)|<\sqrt{\hat c}Q_{2}(c).
\]
On the other hand, if \(\alpha_{m-1}\) is very small, then we have no lower estimate on \(|g_{m}''(z)|\).
However, the next lemma shows that \(\alpha_{n}\) cannot be too small provided that \(c_{n}>1\).
\begin{lemma}
  [{cf.~\cite[Assertion 6]{Teplinskii_Khanin_2004}}]%
  \label{lem:an-ge}
  For any circle homeomorphism \(f\in B^{2+}_{irr}(p, c)\) and any \(\delta>0\), for sufficiently large \(n\) such that \(c_{n}=\hat c>1\) we have
  \begin{equation}
    \label{eq:an-ge}
    \alpha_{n}> \frac{\sqrt{c_{n}}-1}{4}-\delta.
  \end{equation}
\end{lemma}
\begin{proof}
  Since \(f\) has an irrational rotation number, we have \(f_{n}(z)>z\) on \([-1, 0]\) for all \(n\).
  We will use \autoref{thm:dist-fnFn} to rewrite this fact in terms of \(F_{\alpha_{n}, v_{n}, c_{n}}\) and deduce~\eqref{eq:an-ge}.
  Namely, take a large \(n\) such that \(d=\dist_{C[-1, 0]}(f_{n}, F_{\alpha_{n}, v_{n}, c_{n}})<\frac{2\delta}{\sqrt{c_{n}}+1}\) and \(\frac{\sqrt{c_{n}}-1}{2}\le v_{n}\le \sqrt{c_{n}}-1\).
  The numerator of \(F_{\alpha_{n}, v_{n}, c_{n}}(z)-z\) takes its minimal value at \(z_{0}=\frac{1-\sqrt{c_{n}}}{2v_{n}}\).
  Note that \(z_{0}\in\left[-1, -\frac{1}{2}\right]\subset[-1, 0]\), hence
  \[
    F_{\alpha_{n}, v_{n}, c_{n}}(z_{0})\ge f_{n}(z_{0})-d>z_{0}-\frac{2\delta}{\sqrt{c_{n}}+1}.
  \]
  Finally, we substitute formulas for \(z_{0}\) and \(F_{\alpha_{n}, v_{n}, c_{n}}\) into this inequality, and we get \(\alpha_{n}>\frac{{(\sqrt{c_{n}}-1)}^{2}}{4v_{n}}-\delta>\frac{\sqrt{c_{n}}-1}{4}-\delta\).
\end{proof}

Next we provide an explicit estimate on the total distortion of \(Rf_{m}\) for large \(m\).
\begin{lemma}%
  \label{lem:DRfm-lt}
  For any map \(f\in B^{2+}_{irr}(p, c)\) and \(D>\hat c^{4}\), for sufficiently large \(m\in\mathbb N\) the total distortion of \(Rf_{m}\) over the circle~\(\sfrac{[-1, \alpha_{m}]}{(-1\sim \alpha_{m})}\) is less than \(D\),
  \begin{equation}
    \label{eq:DRfm-c2}
    D(Rf_{m})=\exp\left[\var\ln Rf'_{m}(z)\right]<D.
  \end{equation}
\end{lemma}
\begin{proof}
  For simplicity we will only deal with the case \(c_{m}=\hat c>1\).
  The other case is completely analogous.
  Due to \autoref{lem:universal_convexity}, the map \(Rf_{m}\) is convex both on \([-1, 0]\) and \([0, \alpha_{m}]\), hence \(\ln Rf_{m}'\) is a monotonically increasing function on both intervals.
  Next, for sufficiently large \(m\) we have \(f_{m}'(0)>\sqrt{c}\) and \(g_{m}'(0)=f_{m-1}'(0)<\sqrt{c}\).
  This is true because these inequalities hold for \(F_{\alpha_{m}, v_{m}, c_{m}}\) and \(f_{m}\) is close to it.
  Therefore, \(f_{m}'(0)>g_{m}'(0)\). Moreover, \eqref{eq:mul-break} implies have \(f_{m}'(-1)>g_{m}'(\alpha_{m})\).
  Thus
  \begin{align*}
    D(Rf_{m}) &= \frac{f_{m}'(0)}{f_{m}'(-1)} \times \frac{f_{m}'(0)}{g_{m}'(0)} \times \frac{g_{m}'(\alpha_{m})}{g_{m}'(0)} \times \frac{f_{m}'(-1)}{g_{m}'(\alpha_{m})}\\
              &={\left(\frac{f_{m}'(0)}{g_{m}'(0)}\right)}^{2}={\left(\frac{f_{m}'(0)}{f_{m-1}'(0)}\right)}^{2}.
  \end{align*}
  Due to~\eqref{eq:fn-Fn}, it suffices to prove that \(\frac{F_{\alpha_{m}, v_{m}, c_{m}}'(0)}{F_{\alpha_{m-1}, v_{m-1}, c_{m-1}}'(0)}\le \hat c^{2}\) whenever \((\alpha_{m}, v_{m})\in U_{\hat c}\) and \((\alpha_{m-1}, v_{m-1})\in U_{1/\hat c}\).
  This estimate immediately follows from the definitions of \(U_{c}\) and \(F_{\alpha_{m}, v_{m}, c_{m}}\).
\end{proof}

From now on, we fix \(D=D(c)>\hat c^{4}\).
Then for sufficiently large \(m\) we can combine the inequality \(D(Rf_{m})<\hat c^{4}\) with Denjoy inequalities to obtain the following estimates.
First, for any \(x\in I_{m}\), and \(n\ge m\) we have
\begin{equation}
  \label{eq:Im-deriv}
  D^{-1}(c)<D^{-1}(Rf_{m})\le (f^{q_{n}})'(x)\le D(Rf_{m})<D(c)
\end{equation}
whenever \(f^{q_{n}}\) is differentiable at~\(x\).
Second, for any \(x, y \in I_{m}\), and \(k\) such that the intervals \(Rf_{m}^{j}[x..y]\), \(j=0,\dots,k-1\), are pairwise disjoint we have
\begin{equation}
  \label{eq:Im-frac-deriv}
  D^{-1}(c)<D^{-1}(Rf_{m})\le \frac{(Rf_{m}^{k})'(x)}{(Rf_{m}^{k})'(y)}\le D(Rf_{m})<D(c).
\end{equation}

Let us generalize~\eqref{eq:Im-deriv} from \(x\in I_{m}\) to \(x\in S^{1}\).
\begin{lemma}%
  \label{lem:universal_derivative}
  For any \(f\in B^{2+}_{irr}(p, c)\), for sufficiently large \(n\) we have \(D^{-1}(c)< (f^{q_n})'(x) < D(c)\) whenever \(f^{q_{n}}\) is differentiable at~\(x\).
\end{lemma}
Denjoy inequality says that \(D(f)^{-1}\leq (f^{q_n})'(x) \leq D(f)\).
However, this estimate depends on \(f\), and we need a uniform estimate.
\begin{proof}
  Choose \(m\) as in \autoref{lem:DRfm-lt}.
  Due to~\eqref{eq:Im-deriv}, the desired inequality holds on \(I_{m}\).
  Now, for any \(x \in S^{1}\) we have \(f^j(x) \in I_m\) for some \(0\le j<q_{m}\), and
  \[
    (f^{q_n})'(x)= \frac{\prod_{i=0}^{j-1}f'(x_i)}{\prod_{i=0}^{j-1}f'(x_{i+q_n})}(f^{q_n})'(x_j),
  \]
  where \(x_{i}=f^{i}(x)\). Since the upper bound for \(j\) does not depend on \(n\), the first term is arbitrarily close to one for \(n\) large enough, hence we get the required estimate on \({(f^{q_{n}})}'(x)\) for any \(x\).
\end{proof}

Choose an arbitrary point \(x_{0}\) and consider the interval \(\Delta=[x_{0}..f^{q_{n-1}}(x_{0})]\).
Let \(w\) be the renormalized affine coordinate that plays the same role for the point \(x_{0}\) as the coordinate \(z\) for the break point \(p\):
\(w=0\) corresponds to \(x=x_{0}\) and \(w=-1\) corresponds to \(x=f^{q_{n-1}}(x_{0})\).
Let \(h_{n}\) be the map \(f^{q_{n}}\) written in the chart \(w\).
The map \(h_n\) is analogous to \(f_n\) except that the starting point for the dynamical partition is taken with the base point \(x_{0}\) rather than the break point \(p\).

\begin{lemma}%
  \label{lem:everywhere_convexity}
  There exist universal constants \(\hat Q_2(c)>\hat Q_1(c)>0\) such that for all \(f\in B^{2+}_{irr}(p, c)\) for sufficiently large \(n\) the following holds.
  For all \(x_{0}\in S^{1}\) and \(w\in [-1, 0]\) such that \(h_n''(w)\) is defined, we have \(\hat Q_1(c)\leq |h_n''(w)| \leq \hat Q_2(c)\).
  The second derivative \(h_{n}''(w)\) is positive if \(c_{n}>1\) and is negative if \(c_{n}<1\).
\end{lemma}
\begin{proof}
  Recall that \(D(c)\) is a number greater than \(\hat c^{4}\).
  Fix \(m\) such that \(c_{m-1}>1\) and all the estimates from previous lemmas hold for~\(Rf_{m}\).
  In particular, inequalities~\eqref{eq:uniform-convexity},~\eqref{eq:an-ge}, and definition of \(U_{c}\) imply that we have a uniform estimate on the second derivative of \(Rf_{m}\).
  Namely, for every point \(z\in[-1, \alpha_{m}]\) such that \(Rf_{m}''(z)\) is defined we have
  \begin{equation}
    \label{eq:Rfm''}
    -\tilde Q_{2}(c)<Rf_{m}''(z)<-\tilde Q_{1}(c),
  \end{equation}
  where \(\tilde Q_{1}(c)=Q_{1}(c)\min\left(1, \frac{\sqrt{\hat c}-1}{4}\right)\), \(\tilde Q_{2}(c)=Q_{2}(c)\sqrt{\hat c}\).

  First consider the case \(\Delta\subset I_m=\Delta_{m-1}^0\cup \Delta_m^0\).
  Note that \(\Delta\) can be viewed as an element of the dynamical partition for \(Rf_m\) with base point \(z(x_{0})\).
  Let \(B\) be the affine change of variables from \(w\) to \(z\), where \(z\) is the renormalized coordinate on~\(I_m\).
  Then \(h_n= B^{-1}\circ{(Rf_m)}^q\circ B\), where \(q\) is the denominator of the continued fraction approximation for the rotation number of \(Rf_m\) corresponding to \(f^{q_n}\). Denote \(z=B(w)\), \(z_j={(Rf_m)}^{j}z\), \(j\geq 0\). Then
  \[
    h_{n}''(w)=(Rf_m^{q})'(z)\sum_{j=0}^{q-1}\left[\frac{Rf''_m(z_j)}{Rf'_m(z_j)}B'(Rf_m^{j})'(z)\right].
  \]
  The inequalities~\eqref{eq:Im-deriv} and~\eqref{eq:Rfm''} provide us with estimates on \((Rf_{m}^{q})'(z)=\left(f^{q_{n}}\right)'(A_{m}(z))\), \(Rf_{m}''(z_{j})\), and \(Rf_{m}'(z_{j})=(f^{q_{m}})'(A_{m}(z_{j}))\).
  Hence, it suffices to estimate \(B'\sum_{j=0}^{q-1}(Rf_{m}^{j})'(z)\).

  Let \(\Delta'=[B(-1)..B(0)]\) be the interval \(\Delta\) written in the chart~\(z\).
  Let \(l_j\), \(j\geq 0\) be the length of the interval \(Rf_{m}^{j}(\Delta')\), 
  Then \(B'=l_0\) and \((Rf_m^{j})'(z)l_0/l_j\) is the ratio of the derivatives of \(Rf_{m}^{j}\) at two points on \(\Delta'\), hence due to~\eqref{eq:Im-frac-deriv} we have
  \[
    D^{-1}(c)\frac{l_j}{l_0}< (Rf_m^{j})'(z) <D(c)\frac{l_j}{l_0}
  \]
  for all \(j=0,\dotsc,q-1\).

  The pairwise disjoint intervals \(Rf_{m}^{j}(\Delta')\), \(j=0,\dotsc,q-1\), are included by \(I_{m}=[-1, \alpha_{m}]\) while their union with their images under \(Rf_{m}^{q}\) covers this interval.
  Since \(\left(Rf_{m}^{q}\right)'\in({D(c)}^{-1}, D(c))\), we have
  \[
    \frac{1}{1+D(c)}<\frac{1+\alpha_{m}}{1+D(c)}\le \sum_{j=0}^{q-1}l_{j}\le 1+\alpha_{m}\le 1+\sqrt{\hat c},
  \]
  where the outer inequalities follow from \(0<\alpha_{m}\le \sqrt{\hat c}\).
  Thus
  \[
    \frac{\tilde Q_1(c)}{D(c)D^{2}(Rf_{m})(1+D(c))} <|h_{n}''(w)|<\tilde Q_{2}(c)D(c)D^{2}(Rf_{m})(1+\sqrt{\hat c}).
  \]

  Finally, if \(\Delta\) does not belong to \(I_m\), we will choose a number \(s\) such that  \(f^s\Delta\) belongs to \(I_m\) (an upper estimate on \(s\) depends only on \(m\)), and use the above estimate for the map \(f^{q_{n}}\) on \(f^{s} \Delta\). Since for large \(n\) the map that takes a relative coordinate \(w\in [-1,0]\) corresponding to \(\Delta\) into a relative coordinate \(w_s
  \in [-1,0]\) corresponding to \(f^s\Delta\) is close to the identity, we get a required estimate with
  \(\hat Q_1(c)= \frac{\tilde Q_{1}(c)}{D^{3}(c)(1+D(c))}\), \(\hat Q_2(c) = \tilde Q_2(c)D^3(c)(1+\sqrt{\hat c})\) for a sufficiently large \(n\).
\end{proof}

\subsection{Lengths of the intervals \(\Delta_{n}^{k}\)}
In this subsection we use uniform estimates on the derivatives of \(f^{k}\) established above to estimate the lengths of \(\Delta_{n}^{k}\) and prove \autoref{thm:zeta-decay}.
First we use \autoref{lem:C2-hom} to show that all intervals are exponentially small, then we use convexity to show that some intervals of \(\zeta_{n}\) are exponentially small compared to other intervals of the same partition.

\begin{lemma}%
  \label{lem:decay-zeta}
  For \(f \in B^{2+}_{irr}(p, c)\), and for sufficiently large \(n\), the following holds.
  Consider two intervals \(\Delta=\Delta_{n-1}^{k}\), \(0\le k<q_{n}\), and \(\Delta'\in\zeta_{n+2}\), \(\Delta'\subset\Delta\).
  Then
  \[
    \frac{|\Delta'|}{|\Delta|}\le \Lambda(c):=\frac{1}{1+D^{-1}(c)}.
  \]
  Moreover, if \(a_{n+1}>1\), then the same estimate holds for any \(\Delta'\in\zeta_{n+1}\), \(\Delta'\subset \Delta\).
\end{lemma}
\begin{proof}
  Let \(\Delta=\Delta_{n-1}^{k}\) and \(\Delta'=\Delta_{m}^{l}\subset\Delta\), \(m\in\set{n, n+1, n+2}\) be two intervals satisfying assumptions of the lemma.
  It is easy to see that \(\Delta\) includes one of the intervals \(f^{\pm q_{m}}\left(\Delta'\right)\).
  Due to \autoref{lem:universal_derivative}, for sufficiently large \(n\) we have \(\left|f^{\pm q_{m}}\left(\Delta'\right)\right|\ge D^{-1}(c)|\Delta'|\), hence
  \[
    \frac{|\Delta'|}{|\Delta|}\le \frac{|\Delta'|}{|\Delta'|+\left|f^{\pm q_{m}}\left(\Delta'\right)\right|}\le \frac{1}{1+D^{-1}(c)}=\Lambda(c).
  \]
\end{proof}
Iterating the estimate from \autoref{lem:decay-zeta}, we immediately get the first part of \autoref{thm:zeta-decay}.
\begin{corollary}%
  \label{cor:exp-decay-max}
  For any \(f \in B^{2+}_{irr}(p, c)\) we have
  \[
    \max_{0\le k<q_{n}}\left|\Delta_{n-1}^{k}\right|=O\left({\Lambda(c)}^{n/2}\right)
  \]
  as \(n\to\infty\).
\end{corollary}

In the next few lemmas we will prove the second part of \autoref{thm:zeta-decay}.
In order to prove this lemma, we will show that for some constants \(m(c)\) and \(\hat\gamma_{2}(c)<1\) for sufficiently large \(n\), each interval \(\Delta\) of the partition \(\zeta_{n}\) includes a subinterval \(\tilde\Delta\) of the partition \(\zeta_{n+m(c)}\) such that \(\frac{|\tilde\Delta|}{\mu(\tilde{\Delta})}\le\hat\gamma_{2}(c)\frac{|\Delta|}{\mu(\Delta)}\).
This will immediately imply~\eqref{eq:zeta-min-delay} with \(\gamma_{2}(c)=\sqrt[m(c)]{\hat{\gamma}_{2}(c)}\).

In \autoref{large_subinterval} we show that each interval \(\Delta_{n-1}^{k}\) contains a \enquote{long} subinterval \(\Delta_{n}^{l}\), \(|\Delta_{n}^{l}|\ge r(c)|\Delta_{n-1}^{k}|\).
In \autoref{cor:large_a_n} we will use this lemma to construct a \enquote{short} subinterval \(\Delta_{n}^{l}\subset\Delta_{n-1}^{k}\) in the case \(a_{n+1}\ge A(c)\) for some constant \(A(c)\) that will be chosen later.
Finally, in \autoref{several_small_a_n} we will choose \(m(c)\) and find a “short” subinterval \(\tilde\Delta\subset\Delta_{n-1}^{k}\) in the case \(a_{n+j}<A(c)\), \(j=1,\dotsc,m(c)\).
\begin{lemma}%
  \label{large_subinterval}
  There exists a universal constant \(0<r(c)<1\) such that for any \(f\in B^{2+}_{irr}(p, c)\) the following holds.
  For any \(n\) large enough any interval \(\Delta^k_{n-1}\), \(0\leq k <q_n\) contains a \enquote{long} subinterval \(\Delta_n^l\), \(0\leq l <q_{n+1}\) satisfying the following estimate:
  \begin{equation}
    \label{eq:len-ge-rc}
    \frac{|\Delta_n^l|}{|\Delta^k_{n-1}|} \geq r(c).
  \end{equation}
\end{lemma}
\begin{proof}
  Choose \(N(f)\) so large that conclusions of \autoref{lem:universal_derivative} and \autoref{lem:everywhere_convexity} hold for \(n\ge N(f)\), then choose \(n>N(f)\).
  As in \autoref{lem:everywhere_convexity}, let \(w\) be the renormalized affine coordinate corresponding to the interval \(\Delta=\Delta^k_{n-1}\); let \(h\) be the map \(f^{q_{n}}\) written in this chart.
  Since the map \(h\) is uniformly convex (up or down) with at most one break point inside \([-1, 0]\), it cannot be uniformly close to identity.
  In other words, there exists a universal constant \(d(c)>0\) such that \(h_n(w_{0})-w_{0}\geq d(c)\) for some \(w_{0}\in [-1,0]\).
  We claim that~\eqref{eq:len-ge-rc} holds for \(r(c)=\frac{d(c)}{2D(c)}\).

  Let \(x_{0}\in \Delta\) be the point such that \(w(x_{0})=w_{0}\).
  Put \(I=[x_{0}..f^{q_{n}}(x_{0})]\), then \(|I|\ge d(c)|\Delta|\).
  The interval \(I\) has the same invariant measure as each of \(\Delta_{n}^{l}\).
  Consider two cases.

  If \(x_{0}\) belongs to one of the subintervals \(J=\Delta_{n}^{l}\subset\Delta\), then \(I\subset J\cup f^{q_{n}}(J)\).
  Since \(D(c)|J|\ge |f^{q_{n}}(J)|\), we have
  \[
    |J|\ge\frac{|I|}{1+D(c)}\ge\frac{d(c)}{1+D(c)}|\Delta|>\frac{d(c)}{2D(c)}|\Delta|=r(c)|\Delta|.
  \]

  If \(x_{0}\in \Delta_{n+1}^{k}\), then \(I\) is covered by the union of \(\Delta_{n+1}^{k}\) and \(\Delta_{n}^{k}\), thus one of these two intervals has length at least \(\frac{|I|}{2}\).
  Note that \(f^{-q_{n}}(\Delta_{n+1}^{k})\subset\Delta_{n}^{k+q_{n+1}-q_{n}}\) and \(f^{q_{n-1}}(\Delta_{n}^{k})=\Delta_{n}^{k+q_{n-1}}\).
  Since the derivatives of \(f^{q_{n}}\) and \(f^{q_{n-1}}\) belong to the interval \([D^{-1}(c), D(c)]\), one of the intervals \(\Delta_{n}^{k+q_{n+1}-q_{n}}\subset\Delta\), \(\Delta_{n}^{k+q_{n-1}}\subset\Delta\) has length at least \(\frac{|I|}{2D(c)}\ge \frac{d(c)}{2D(c)}|\Delta|=r(c)|\Delta|\).
\end{proof}
The following corollary follows immediately from the previous lemma.
\begin{corollary}%
  \label{cor:large-interval}
  For each \(f\in B^{2+}_{irr}(p, c)\) there exists a constant \(C(f)>0\) such that \(\max_{0\leq k<q_n}{|\Delta_n^k|}\geq C(f)r^n(c)\)
\end{corollary}
\begin{lemma}%
  \label{lem:frac-len-meas-lt}
  For any \(f\in B^{2+}_{irr}(p, c)\), for sufficiently large \(n\) such that \(a_{n+1}>1\) the following holds.
  For each interval \(\Delta_{n-1}^{k}\), \(0\le k<q_{n}\), there exists a subinterval \(\Delta_{n}^{l}\subset\Delta_{n-1}^{k}\), \(0\le l<q_{n+1}\), such that
  \[
    \frac{|\Delta_{n}^{l}|}{|\Delta_{n-1}^{k}|}\le (1-r(c))\frac{a_{n+1}+1}{a_{n+1}-1}\times\frac{\mu(\Delta_{n}^{l})}{\mu(\Delta_{n-1}^{k})}.
  \]
\end{lemma}
\begin{proof}
  The interval \(\Delta_{n-1}^{k}\) includes \(a_{n+1}\) subintervals \(\Delta_{n}^{l}\), \(0\le l<q_{n+1}\).
  Due to \autoref{large_subinterval}, one of them has length at least \(r(c)|\Delta_{n-1}^{k}|\), hence one of the other \(a_{n+1}-1\) subintervals has length at most \(\frac{1-r(c)}{a_{n+1}-1}|\Delta_{n-1}^{k}|\).
  On the other hand, \(\mu(\Delta_{n}^{l})\ge\frac{\mu(\Delta_{n-1}^{k})}{a_{n+1}+1}\).
  Dividing these two inequalities, we obtain the required estimate.
\end{proof}
Since \(\frac{a_{n+1}+1}{a_{n+1}-1}\to 1\) as \(a_{n+1}\to\infty\), we have the following estimate.
\begin{corollary}%
  \label{cor:large_a_n}
  For any \(\gamma>1-r(c)\) there exists \(A(c, \gamma)\) such that for any \(f\in B^{2+}_{irr}(p, c)\) for sufficiently large \(n\) such that \(a_{n+1}\ge A(c, \gamma)\) the following holds.
  For each interval \(\Delta_{n-1}^{k}\), \(0\le k<q_{n}\), there exists a subinterval \(\Delta_{n}^{l}\subset\Delta_{n-1}^{k}\), \(0\le l<q_{n+1}\), such that
  \[
    \frac{|\Delta_{n}^{l}|}{|\Delta_{n-1}^{k}|}\le \gamma\frac{\mu(\Delta_{n}^{l})}{\mu(\Delta_{n-1}^{k})}.
  \]
\end{corollary}

From now on, we fix some \(\gamma(c)\in(1-r(c), 1)\) and put \(A(c)=A(c, \gamma(c))\).
Consider the case when \(a_{n+1}\) is smaller than \(A(c)\).
We start with a simple observation.

\begin{lemma}%
  \label{lem:frac-len-succ-ge}
  Consider a map \(f\in B^{2+}_{irr}(p, c)\).
  Let \(n\) be so large that the conclusion of \autoref{lem:universal_derivative} holds for \(f^{q_{n-1}}\) and \(f^{q_{n}}\).
  Then for each \(k=0,\dotsc,q_{n}-1\) we have
  \begin{gather*}
    |\Delta_{n-1}^{k}| \le C(D, a_{n+1})|\Delta_{n}^{k}|,\\
    \intertext{where}
    C(D, a_{n+1})=D\left(1+\sum_{i=0}^{a_{n+1}-1}D^{i}\right).
  \end{gather*}
\end{lemma}
\begin{proof}
  This lemma immediately follows from the inclusion
  \[
    \Delta_{n-1}^{k}\subset f^{-q_{n}}(\Delta_{n}^{k})\cup \bigcup_{i=0}^{a_{n+1}-1}f^{q_{n-1}+iq_{n}}(\Delta_{n}^{k})
  \]
  and the estimates on the derivatives of \(f^{q_{n}}\) and \(f^{q_{n-1}}\).
\end{proof}
\begin{lemma}%
  \label{lem:small_a_n}
  There exist two universal positive constants \(\alpha(c)\), \(\beta(c)\) such that for any \(f\in B^{2+}_{irr}(p, c)\) for sufficiently large \(n\) the following holds.
  Suppose \(a_{n+1}\leq A(c)\), \(a_{n+2}\leq A(c)\).
  Then for each \(\Delta = \Delta^k_{n-1}\) one can find a closed subinterval \(\hat\Delta\subset\Delta\cap f^{-q_{n}}(\Delta)\) of length at least \(|\hat\Delta|\ge \alpha(c)|\Delta|\) such that \(\left|\log {(f^{q_{n}})}'(x)\right|\ge \beta(c)\) for \(x\in\hat\Delta\).
\end{lemma}
\begin{proof}
  First we prove a lower estimate on the length of the interval \(I=\Delta\cap f^{-q_{n}}(\Delta)=[f^{k+q_{n-1}}(p)..f^{k-q_{n}}(p)]\).
  Since \(a_{n+1}\le A(c)\) and \(a_{n+2}\le A(c)\), \autoref{lem:frac-len-succ-ge} implies that \(|\Delta_{n}^{k}|\ge C^{-1}(D(c), A)|\Delta|\) and \(|\Delta_{n+1}^{k}|\ge C^{-2}(D(c), A)|\Delta|\).
  Note that \(\Delta_{n+1}^{k}\subset f^{q_{n}}(I)\), hence \(|I|\ge \frac{|\Delta|}{D(c)C^{2}(D(c), A)}\).

  Now we divide \(I\) into three subintervals of equal length.
  We state that one of these subintervals satisfies all the requested properties.
  Each of them automatically satisfies the properties \(\hat\Delta\subset\Delta\cap f^{-q_{n}}(\Delta)\) and \(|\hat\Delta|\ge \alpha(c)|\Delta|\) with \(\alpha(c)=\frac{1}{3D(c)C^{2}(D(c), A)}\), so it suffices to prove that \(|\ln (f^{q_{n}})'(x)|\) is bounded away from zero on one of these subintervals.
  
  Note that \((f^{q_{n}})'(x)=h_{n}'(w(x))\), so it suffices to prove a similar estimate for \(h_{n}\).
  Using convexity of \(h_n\) we conclude that the estimate holds for either the left or the right of these subintervals with \(\beta(c)=\ln{(1+\hat{Q}_1\alpha(c)/2)}\).
\end{proof}
Assume now that several partial quotients \(a_{n+1}, a_{n+2}, \dots, a_{n+m}\) are all smaller than \(A(c)\).
\begin{lemma}%
  \label{several_small_a_n}
  There exist universal constants \(m(c) \in \bbN\) and \(0<\gamma_2'(c)<1\) such that for each \(f\in B^{2+}_{irr}(p, c)\) if \(n\) is large enough and \(a_{n+i}\leq A(c)\), \(1\leq i \leq m(c)+1\), then for any \(\Delta = \Delta^k_{n-1}\), \(0\leq k <q_n\), one can find a subinterval \(\tilde\Delta \subset \Delta\) which is an element of the partition \(\zeta_{n+m(c)}\) such that
  \[
    \frac{|\tilde\Delta|}{|\Delta |}\leq \gamma_2'(c)\frac{\mu(\tilde\Delta)}{\mu(\Delta)}.
  \]
\end{lemma}
\begin{proof}
  Choose \(m(c)\) so that \({\Lambda(c)}^{\lfloor\frac{m(c)}{2}\rfloor}<\frac{\alpha(c)}2\), where \(\Lambda(c)\) is the same as in \autoref{lem:decay-zeta}, then choose \(N=N(f)\) such that for \(n\ge N\) the estimates from all previous lemmas hold true.
  A partitioning of any interval \(\Delta = \Delta^k_{n-1}, 0\leq k <q_n\) onto smaller subintervals corresponding to the partition \(\zeta_{n+m(c)}\) generates two finite conditional probability distributions: \(\mathbf P=(\mathbf p_i, 1\leq i \leq M)\) corresponding to the Lebesgue measure, and \(\mathbf Q=(\mathbf q_i, 1\leq i \leq M)\) related to the invariant measure \(\mu\).
  Since \(a_{n+i}\leq A(c)\), \(1\leq i \leq m(c)+1\), we have \(\mathbf q_{i}\ge \eps(c)={(A+1)}^{-m(c)-1}\), \(i=1,\dotsc,M\), and \(M\leq1/\eps(c)\).
  We have to show that there exists \(0<\gamma_2'(c)<1\) such that \(\mathbf p_i\leq\gamma_2'(c)\mathbf q_i\) for at least one index \(i\).
  Denote \(\hat{\Delta}\) be the subinterval of \(\Delta\) provided by \autoref{lem:small_a_n}.
  Due to the choice of \(m(c)\) and \autoref{lem:decay-zeta}, the lengths of all elements of the partition \(\zeta_{n+m(c)}\) inside \(\Delta\) are less than \(\frac{\alpha(c)}{2}|\Delta|\), hence one of these elements \(\Delta'\) is included by \(\hat{\Delta}\).
  Since the derivative \((f^{q_{n}})'\) is uniformly bounded away from 1 on \(\hat\Delta\), we have \(\left|\ln{\frac{|f^{q_{n}}(\Delta')|}{|\Delta'|}}\right|\geq \beta(c)\).
  Hence, there exists a pair \(1\leq i,j\leq M\) such that \(\frac{\mathbf p_i}{\mathbf p_j}\geq \exp{(\beta(c))}\) and \(\mathbf q_i=\mathbf q_j\).

  Since \(\frac{\mathbf p_{i}/\mathbf q_{i}}{\mathbf p_{j}/\mathbf q_{j}}\ge \exp{(\beta(c))}\), we have either \(\frac{\mathbf p_{j}}{\mathbf q_{j}}\le \exp{\left(-\frac{\beta(c)}{2}\right)}\), or \(\frac{\mathbf p_{i}}{\mathbf q_{i}}\ge \exp{\left(\frac{\beta(c)}{2}\right)}\).
  In the former case we already have the desired estimate.
  In the latter case the inequalities \(1>\mathbf p_{i}\ge \exp{\left(\frac{\beta(c)}{2}\right)}\mathbf q_{i}\) and \(\mathbf q_{i}\ge \eps\) imply
  \[
    \frac{\sum_{s\ne i}\mathbf p_{s}}{\sum_{s\ne i}\mathbf q_{s}}=\frac{1-\mathbf p_{i}}{1-\mathbf q_{i}}\le \frac{1-\eps\exp{\left(\frac{\beta(c)}{2}\right)}}{1-\eps},
  \]
  hence one of the ratios \(\frac{\mathbf p_{s}}{\mathbf q_{s}}\), \(s\ne i\), is less than or equal to the same constant.
\end{proof}

Now we are ready to prove the second part of \autoref{thm:zeta-decay}.
\begin{lemma}%
  \label{small_intervals}
  There exists a universal constant \(0<\gamma_2(c)<1\) such that for every \(f\in B^{2+}_{irr}(p, c)\) we have
  \[
    \min_{0\leq k<q_n}{|\Delta_{n-1}^k|}=O(\gamma_2^n(c)\mu_{n-1}),
  \]
  as \(n\to\infty\), where \(\mu_{n-1}= |q_{n-1}\rho-p_{n-1}|\).
\end{lemma}
\begin{proof}
  Every interval \(\Delta_{n-1}^k\) can be obtained by a sequence of partitions. On every step of this sequence a given interval
  is divided onto a finite number of subintervals. As above we have two finite conditional  probability distributions
  corresponding to the Lebesgue measure and the invariant measure \(\mu\). It follows immediately
  that on every step we can always find a subinterval \(\Delta'\) of \(\Delta\) such that \(\frac{|\Delta'|}{|\Delta |}\leq \frac{\mu(\Delta')}{\mu(\Delta)}\). However, on many steps of partitioning \(\frac{|\Delta'|}{|\Delta |}\) will be constant times smaller than
  \(\frac{\mu(\Delta')}{\mu(\Delta)}\). Indeed, by \autoref{cor:large_a_n} it happens every time when \(a_i\geq A(c)\).
  By \autoref{several_small_a_n} the same is true when we have a sequence of length \(m(c)+1\) of consecutive \(a_i\leq A(c)\).
  We can divide all the steps into blocks of length \(m(c)+1\). If the block contains at least one \(a_i\geq A(c)\) we have contraction of the ratio
  by the first reason, and if all \(a_i\leq A(c)\) contraction is related to the second mechanism. Hence, since the number of blocks is of
  the order of \(\sfrac{n}{(m(c)+1)}\), the statement of the lemma holds.
\end{proof}

\section{Proof of the main theorem}\label{sec:proof-main}

The main idea of the proof is very simple.
Let \(\frac{p_{n}}{q_{n}}\) be the \(n\)-th continued fraction convergent of \(\rho(f)=\rho(g)\), let \(J\) be the smallest of the intervals \(\Delta^k_{n-1}\), \(0\leq k <q_n\).
If \(h\) is \(C^{2}\) smooth, then due to \autoref{lem:sum-sqr-le} and \autoref{lem:universal_derivative} we have \(\sum_{k=0}^{q_{n}-1}|f^{k}(J)|^{2}=o(|J|)\).
On the other hand, \(\sum_{k=0}^{q_n-1} |\Delta_{n-1}^k|^2 \geq |J|\sum_{k=0}^{q_n-1} |\Delta_{n-1}^k|\geq C|J|\). This, skipping very minor technical details, gives a contradiction which proves that  \(h\) cannot be \(C^2\) smooth. However, since we want to show that the smoothness of \(h\)
is strictly weaker than \(C^2\) in the Hölder sense, the proof must be more technical, and we need estimates from \autoref{thm:zeta-decay}.

From now on, we assume that \(n\) is so large that all estimates from the lemmas in \autoref{sec:decay-interv-part} hold true.
Let \(\breve\Delta_{n}=\Delta_{n-1}^{l_{n}}\) be the smallest of the intervals \(\Delta_{n-1}^{k}\), \(0\le k<q_{n}\).
The point \(f^{l_{n}-q_{n}}(p)\) splits~\(\breve\Delta_{n}\) in two subintervals.
Let \(J_{n}\) be one of these subintervals with the larger value of \(\sum_{k=0}^{q_n-1}{\left|f^{k}(J_n)\right|}^{2}\).
Then 
\begin{equation}
  \label{eq:sum-sqr-le-sum-sqr}
  \sum_{k=0}^{q_n-1}{\left|f^{k}(J_n)\right|}^{2}\ge\frac{1}{2}\sum_{k=0}^{q_n-1}{\left|f^{k}(\breve\Delta_n)\right|}^{2}\ge \frac{1}{2D(c)}\sum_{k=0}^{q_{n}-1}\left|\Delta_{n-1}^{k}\right|^{2}.
\end{equation}

Note that none of the intervals \(f^{k}(int(J_{n}))\), \(0\le k<q_{n}\), contains~\(p\). Due to \autoref{cor:exp-decay-max} and \autoref{lem:universal_derivative},
\[
  \max_{0\le k<q_{n}}\left|f^{k}(J_{n})\right|\to0
\]
as \(n\to\infty\), thus we can apply \autoref{lem:sum-sqr-le} to \(J_{n}\).
If \(h\) is \(C^{1+\alpha}\), \(0<\alpha<1\), then we have
\begin{align*}
  \sum_{k=0}^{q_{n}-1}\left|\Delta_{n-1}^{k}\right|^{2}&=O\left(\sum_{k=0}^{q_n-1}{\left|f^{k}(J_n)\right|}^{2}\right) &&\text{due~to~\eqref{eq:sum-sqr-le-sum-sqr}}\\
                                                       &=O\left({|J_n|}^{\alpha}+{|f^{q_n}(J_n)|}^{\alpha}\right) &&\text{due~to~\autoref{lem:sum-sqr-le}}\\
                                                       &=O\left({|J_n|}^{\alpha}\right) &&\text{due~to~\autoref{lem:universal_derivative}}\\
                                                       &=O\left({\left|\breve\Delta_n\right|}^{\alpha}\right) && \text{since \(J_{n}\subset\breve\Delta_{n}\)}\\
                                                       &=O\left(\gamma_{2}(c)^{n\alpha}|q_{n-1}\rho-p_{n-1}|^{\alpha}\right) && \text{due~to~\autoref{small_intervals}}\\
                                                       &=O\left(\frac{\gamma_{2}(c)^{n\alpha}}{q_{n}^{\alpha}}\right) && \text{since \(|q_{n-1}\rho-p_{n-1}|<\frac{1}{q_{n}}\)}
\end{align*}

Now let us prove a lower estimate on \(\sum_{k=0}^{q_{n}-1}\left|\Delta_{n-1}^{k}\right|^{2}\) that will contradict this estimate for \(\alpha\) sufficiently close to \(1\).
First, due to \autoref{cor:large-interval} we have
\[
  r^{2n}(c)=O\left(\max_{0\le k<q_{n}}\left|\Delta_{n-1}^{k}\right|^{2}\right)=O\left(\sum_{k=0}^{q_{n}-1}\left|\Delta_{n-1}^{k}\right|^{2}\right).
\]
Next, Cauchy inequality implies that
\[
  \sum_{k=0}^{q_{n}-1}\left|\Delta_{n-1}^{k}\right|^{2}\ge\frac{1}{q_{n}}\left(\sum_{k=0}^{q_{n}-1}\left|\Delta_{n-1}^{k}\right|\right)^{2}\ge\frac{1}{(1+D(c))^{2}}\frac{1}{q_{n}},
\]
where the last inequality follows from \(S^{1}=\bigcup_{k=0}^{q_{n}-1}\left(\Delta_{n-1}^{k}\cup f^{q_{n}}\left(\Delta_{n-1}^{k}\right)\right)\) and \autoref{lem:universal_derivative}.

Combining lower and upper estimates on \(\sum_{k=0}^{q_{n}-1}\left|\Delta_{n-1}^{k}\right|^{2}\), we obtain
\begin{align*}
  r^{2n}(c)&=O\left(\frac{\gamma_{2}(c)^{n\alpha}}{q_{n}^{\alpha}}\right) &\frac{1}{q_{n}}&=O\left(\frac{\gamma_{2}(c)^{n\alpha}}{q_{n}^{\alpha}}\right).
\end{align*}
Take \(\nu(c)=\frac{\ln\gamma_{2}(c)}{2\ln r(c)+\ln \gamma_{2}(c)}\), then we have
\[
  \left(\frac{\gamma_{2}(c)^{n}}{q_{n}}\right)^{1-\nu}=\left(r^{2n}(c)\right)^{\nu}\left(\frac{1}{q_{n}}\right)^{1-\nu}=O\left(\frac{\gamma_{2}(c)^{n\alpha}}{q_{n}^{\alpha}}\right),
\]
hence \(\alpha\le 1-\nu\). This completes the proof of \autoref{thm:main}.

\section{Conjugacy of fractional linear maps}\label{sec:conj-fract-line}
\autoref{thm:main} shows that in some cases two maps of class \(B^r(p, c)\) with the same rotation number are not \(C^{2}\) conjugate.
We expect that this is true in many more cases.
In this section we discuss the case of linear fractional maps.

Recall (see \cites[Proposition 4.6 and Theorem 2]{Khanin_Vul_1991}[Assertions 1, 2]{Teplinskii_Khanin_2004}[Theorem 4]{Ghazouani_Khanin_2021}) that renormalizations of a map \(f\in B^{r}(p, c)\) tend to a pair of linear fractional maps of the form
\begin{align}
  \label{eq:FG}
  F_{\alpha,v,c}(z)&=\frac{\alpha+\sqrt{c}z}{1-vz}, & G_{\alpha, v, c}(z)&=\frac{\alpha(z-\sqrt{c})}{\alpha\sqrt{c}+z(1+v-\sqrt{c})}, &(\alpha, v)&\in U_{c}.
\end{align}
We used one half of this statement earlier, see \autoref{sec:decay-interv-part}, \autoref{thm:dist-fnFn}.
It is easy to see that both \(F_{\alpha,v,c}\) and \(G_{\alpha,v,c}\) are increasing maps, \(F_{\alpha, v, c}(0)=\alpha\), \(G_{\alpha, v, c}(0)=-1\), and \(F_{\alpha,v,c}(-1)=G_{\alpha,v,c}(\alpha)=\frac{\alpha-\sqrt{c}}{1+v}\), hence the map \(T_{\alpha,v,c}\colon \sfrac{[-1, \alpha]}{(-1\sim \alpha)}\) given by \(F_{\alpha,v,c}\) on \([-1, 0]\) and by \(G_{\alpha,v,c}\) on \([0, \alpha]\) is a well-defined homeomorphism of the circle, cf. \(Rf_{m}\) in \autoref{sec:decay-interv-part}.
Maps \(T_{\alpha, v, c}\) with a fixed irrational rotation number belong to a smooth curve on the \((\alpha,v)\) plane, see~\cite{Khanin_Teplinsky_2013}.

Because of the special role of the family~\eqref{eq:FG}, it is natural to ask whether two different maps \(T_{\alpha,v,c}\) can be \(C^{r}\) conjugate to each other with high values of \(r\).
Arguments of \autoref{thm:main} do not apply in this case since the Schwarzian derivative vanishes for Möbius maps.
However, the following argument suggested by Selim Ghazouani allows us to show that two different maps \(T_{1}=T_{\alpha_{1}, v_{1}, c}\) and \(T_{2}=T_{\alpha_{2}, v_{2}, c}\) with an irrational rotation number cannot be conjugate by a piecewise \(C^{3}\)-smooth map that sends zero to zero.

The latter condition is technical and probably can be omitted but it is natural to impose this restriction and it greatly simplifies the proof.

Assume that \(h\circ T_{1}=T_{2}\circ h\), where \(h\) is a piecewise \(C^{3}\)-smooth homeomorphism of the circle, \(h(0)=0\).
First we show that \(h\) has to be a piecewise fractional linear map.
Indeed, for any \(x\) such that all the derivatives below exist, we have
\[
  S(h)(T_{1}(x)){(T_{1}(x))}^{2}=S(h\circ T_{1})(x)=S(T_{2}\circ h)=S(h).
\]
Note that \(\rho(x)=\sqrt{|S(h)(x)|}\) satisfies the equation \(\rho(T_{1}(x))=\frac{1}{T_{1}'(x)}\rho(x)\) at all but finitely many points, hence \(\rho\) is the density of a \(T_{1}\)-invariant measure.
It is well-known that a circle maps with breaks have singular invariant measures, thus \(\rho=0\), therefore \(h\) is a piecewise fractional linear map.

Now let us show that \(T_{1}=T_{2}\).
Instead of dealing with these maps, we will compare their \(n\)-th renormalizations for sufficiently large \(n\).
The renormalization operator sends \((F_{\alpha,v,c}, G_{\alpha,v,c})\) to a pair of the same form with \(1/c\) playing the role of \(c\).
Note that the renormalization operator is invertible on the space of maps \(T_{\alpha, v, c}\): if two maps \(T_{\alpha_{1}, v_{1}, c}\) and \(T_{\alpha_{2}, v_{2}, c}\) have the same irrational rotation number and their \(n\)-th renormalizations are equal for some \(n\), then the original maps are equal as well.

Fix a sufficiently large number \(n\) such that each of the intervals \([0..T_{i}^{q_{n-1}}(0)]\), \([0..T_{i}^{q_{n}}]\), \(i=1,2\), is contained in a single interval of smoothness of the conjugacy \(h\) or \(h^{-1}\).
The renormalized maps are Möbius transformations.
Let \(A_{i}=\left(\begin{smallmatrix}\sqrt{c}&a_{i}\\-w_{i}&1\end{smallmatrix}\right)\), \(B_{i}=\left(\begin{smallmatrix}a_{i}&-a_{i}\sqrt{c}\\1+w_{i}-\sqrt{c}&a_{i}\sqrt{c}\end{smallmatrix}\right)\) be the matrices that represent the renormalized maps; let \(C\) be the matrix that represents the restriction of the renormalization of \(h\) to \([-1, 0]\).
Then we have \(CA_{1}C^{-1}\sim A_{2}\) and \(CB_{1}A_{1}C^{-1}\sim B_{2}A_{2}\), where \(\sim\) means that two matrices are proportional.
Indeed, in both cases the corresponding maps coincide on a nontrivial interval: it is \([-1, 0]\) in the first case and \([A_{2}^{-1}(0),0]\) in the second case.

Therefore, \(CA_{1}C^{-1}\sim A_{2}\) and \(CB_{1}C^{-1}\sim B_{2}\).
Since \(h(0)=0\), we have \(C\left(\begin{smallmatrix}0\\1\end{smallmatrix}\right)\sim \left(\begin{smallmatrix}0\\1\end{smallmatrix}\right)\).
Next, \(B_{i}\left(\begin{smallmatrix}0\\1\end{smallmatrix}\right)\sim \left(\begin{smallmatrix}-1\\1\end{smallmatrix}\right)\), hence \(C\left(\begin{smallmatrix}-1\\1\end{smallmatrix}\right)\sim \left(\begin{smallmatrix}-1\\1\end{smallmatrix}\right)\).
Therefore, after multiplication by a constant we may assume that \(C=\left(\begin{smallmatrix}1&0\\u&u+1\end{smallmatrix}\right)\).
Note that \(\Tr(CA_{1}C^{-1})=\Tr(A_{1})=1+\sqrt{c}=\Tr(A_{2})\) and \(CA_{1}C^{-1}\sim A_{2}\) imply that \(CA_{1}C^{-1}=A_{2}\).
The top left elements of these matrices are equal to \(\sqrt{c}-\frac{a_{1}u}{u+1}\) and \(\sqrt{c}\), respectively.
Since \(a_{1}\ne 0\), we get \(u=0\), thus \(A_{1}=A_{2}\), so \(a_{1}=a_{2}\) and \(w_{1}=w_{2}\).

Finally, \(n\)-th renormalizations of \(T_{1}\) and \(T_{2}\) are equal to each other, hence the original maps are equal as well.

\section{Concluding remarks and open problems}\label{sec:concluding}
\subsection{Circle homeomorphisms with breaks}\label{sec:future-breaks}
While \autoref{thm:main} states that a map \(f\) with \(S(f)<0\) cannot be \(C^{2-\nu}\) conjugate to a linear fractional map \(g\), we expect that two generic maps \(f\), \(g\), \(\rho(f)=\rho(g)\notin\bbQ\), are not \(C^{2}\) conjugate to each other.
More precisely, we expect that asymptotically we have
\[
  \left|\xi_{f^{q_{n}}}(J_{n})-  \xi_{g^{q_{n}}}(J'_{n})\right|\ge C\sum_{k=0}^{q_{n}-1}{\left|f^{k}(J_{n})\right|}^{2},
\]
where \(J'_n\) is the interval conjugated to \(J_n\), and one can repeat the proof of \autoref{thm:main}.

For different \(k\geq 2\) one can ask what is the codimension of maps which are \(C^k\) conjugate to a given generic map.

It is also natural to ask about the smoothness of conjugacy for linear fractional maps, see \autoref{sec:conj-fract-line}.

\subsection{Critical circle maps}\label{sec:critical-circle-maps}

The main difficulty in carrying out the same program for critical circle maps is that the Schwartzian derivative of a critical circle map is infinite at a critical point. As a result it is much more difficult to estimate the distortion of the cross-ratio.
Nevertheless we still expect that generically the difference between the total cross-ratio distortions for two different critical circle maps
is of the order \(\sum_{k=0}^{q_n-1}|f^k(J_n)|^2\). Another ingredient of our analysis, namely universal bounds, is readily available
in the critical setting. The arguments above gives us a reason to believe that the following conjecture holds.
\begin{conjecture}
  For any \(\gamma>1\) there exists a constant \(0<\nu(\gamma)<1\) such that for any
  irrational \(0<\rho<1\) one can find a pair of smooth \(\gamma\)-critical circle maps
  \(f,g\) with the rotation number \(\rho\) for which there is no conjugacy of class
  \(C^{2-\delta}\), where \(\delta<\alpha\).
\end{conjecture}
Smoothness for \(f\), \(g\) means here \(C^{\infty}\) outside of critical points, or analytic in the case of odd integer \(\gamma\).
\printbibliography
\end{document}